\documentclass[12pt,a4paper]{amsart}
\usepackage{amsmath,amsfonts}
\usepackage{graphicx}
\usepackage{epsfig}
\usepackage{amssymb}
\usepackage{hyperref}

\newtheorem{theorem}{Theorem}[section]

\newtheorem{cor}[theorem]{Corollary}

\newtheorem{property}[theorem]{Property}
\newtheorem{proposition}[theorem]{Proposition}
\newtheorem{lemma}[theorem]{Lemma}
\newtheorem{definition}[theorem]{Definition}
\newtheorem{defprop}[theorem]{Definition and Proposition}

\newtheorem*{decotheorem}{Decoration Theorem}

\newcommand{\C}{\mathbb{C}}
\newcommand{\JJ}{\mathcal J}
\newcommand{\LL}{\mathcal L}
\newcommand{\MM}{\mathcal M}
\newcommand{\MMp}{\mathcal M'}
\renewcommand{\AA}{\mathcal A}
\newcommand{\RR}{\mathcal R}
\newcommand{\VV}{\mathcal V}
\newcommand{\YYp}{{\mathcal Y'}}
\newcommand{\XX}{\mathcal X}
\newcommand{\YY}{\mathcal Y}
\newcommand{\ZZ}{\mathcal Z}

\newcommand{\eps}{\varepsilon}
\newcommand{\sm}{\setminus}
\newcommand{\ovl}{\overline}
\newcommand{\intr}{\operatorname{int}}

\begin{document}

\title[The Decoration Theorem]{The Decoration Theorem for Mandelbrot and Multibrot Sets \\ }
\author{Dzmitry Dudko}
\date{February 25, 2010}
\maketitle

\begin{abstract}
We prove the decoration theorem for the Mandelbrot set (and Multibrot sets) which says that
 when a ``little Mandelbrot set'' is removed from the Mandelbrot set, then most of the
  resulting connected components have small diameters.
\end{abstract}

\section{Introduction}
 The Mandelbrot set $\MM$ is defined as the set of
quadratic polynomials $z^2+c$  with connected Julia sets. It is a
compact, connected, and full set, and in addition, it has a rich
combinatorial structure. Furthermore, the Mandelbrot set is
self-similar in a certain sense: there are infinitely many subsets
that, together with the induced combinatorics, are canonically
homeomorphic to $\MM$; these are called \emph{small copies of $\MM$
within $\MM$}.

There is a classification of all small copies of the Mandelbrot set;
every copy has particular dynamical properties, hence all copies can
be defined and distinguished in the dynamical planes (see
Definition~\ref*{defMM}).

   In this paper, we prove the following theorem which was conjectured in the
mid-1990 by Mikhail Lyubich and Dierk Schleicher, as well as by
Carsten Lunde Petersen:

\begin{decotheorem}
\label{MainConj} Let $\MM$ be the Mandelbrot set and let $\MM_{s}$
be a small copy of the Mandelbrot set. Then for any $\varepsilon >
0$, there are at most finitely many connected components of $\MM
\backslash \MM_{s}$ with diameter at least $\eps$.
\end{decotheorem}
The countably many components of $\MM\sm\MM_s$ are called the
decorations of $\MM_s$, and the claim is that most of them are
small.

Our main tool will be puzzle and parapuzzle theory.

\medskip

\noindent\textbf{Remark.} The entire construction and the proof will
be carried out for the Mandelbrot set, but they work just the same
for all Multibrot sets
\[
\MM_d:=\{c\in\C\colon \text{the Julia set of $z\mapsto z^d+c$ is connected}\}
\]
for $d\ge 2$. We refrained from working out the details for simplicity of notation.
More precisely, everything we are doing uses conformal pull-backs of a single annulus (in two different cases);
 we do not encounter problems where the combinatorics grows more slowly than
 the shrinking of moduli for high-degree pull-backs. However, we do use the result
 that all parameters $c\in\MM_d$ that are not infinitely renormalizable have trivial
 fibers (so that $\MM_d$ is locally connected at these parameters). This was proved by Yoccoz for
 $d=2$ and by Kahn and Lyubich for all $d\ge 2$.

The Decoration Theorem thus holds for all degree $d\ge2$.

\medskip

\noindent\textbf{Remark.} The entire construction is local in the
sense that it not only works for the Mandelbrot set, but also for
full families of quadratic-like maps (and similarly for
Multibrot-like maps). The details are quite similar to the text as
written and are omitted.

\subsection{Terminology and Notation:}

$f_{c}(z)=z^{2}+c$ is a quadratic polynomial.

The level of equipotentials will be called \textit{height}, and the
\textit{depth} of a puzzle piece is the number of iterations it
takes to map the puzzle piece to a piece of the initial puzzle.

For every puzzle piece the upper index is its depth. If a puzzle
piece is ``unique'', then the subindex will be $0$ or $1$ depending
on the context. For example, $Y^n_0$ contains the critical point
while $Y^n_1$ contains the critical value.

  \textit{We will use the following conventions}: objects in the parameter
plane will be denoted by calligraphic capital letters (such as
$\MM$, $\ZZ^n_i$) while those in dynamical planes will be denoted by
Roman capitals (such as $Z'^m_j$, $Y^p_1$).

  We will use floor brackets (for example,
$\lfloor Y^p_1\rfloor$ or $\lfloor\ZZ'^m_j\rfloor$) to untruncate
the corresponding puzzle or parapuzzle piece.

Following tradition, we slightly abuse (and thus simplify) notation
and use the modulus of an annulus $A$ even when $A$ is not open,
provided its boundary is piecewise smooth (which will always be the
case for annuli constructed by puzzle pieces).

By ``combinatorics'' we mean the angles and heights of rays and
equipotentials. In particular, two puzzle pieces in different planes
are (combinatorially) the same if there is a homeomorphism of their
boundaries sending rays and equipotentials to rays and
equipotentials with equal angles and heights.

In the paper all renormalizations are simple; we will not consider
crossed renormalizations.

\subsection{Outline of the paper.}
  In Section~\ref{sec1} we will briefly review the puzzle (and
  parapuzzle)
  construction. We also will fix some conventions.

     Section~\ref{CombinSect} contains the combinatorics. We will
     discuss the relation between small copies of the Mandelbrot set and their decorations, as well as with puzzle
     pieces in the dynamical planes.

     First we reformulate
  the problem in terms of puzzle and parapuzzle pieces. Every
  decoration is inside a parapuzzle piece $\ZZ^n_i$ associated to that decoration.

     If the Decoration Theorem was not true, then
  big decorations must accumulate at some point $c_0$ from the Mandelbrot
  set. The aim is to show that for every $\ZZ^n_i$ sufficiently close to $c_0$
  there exists an annulus $\AA^n_i$ such that the following
  properties hold:
\begin{itemize}
\item $\AA^n_i$ surrounds $\ZZ^n_i$, but neither contains nor surrounds $c_0$;

\item the moduli of $\AA^n_i$ are uniformly bounded below.
\end{itemize}
  This will conclude the proof.

The first observation is that $c_0$ must be so that the fiber of
$\MM$ at $c_0$ is not trivial. Therefore, by Yoccoz's results it is
enough to consider the case when $z^2+c_0$ is an infinitely
renormalizable polynomial. Hence $c_0\in \MM'_s\subsetneq\MM_s$,
where $\MM'_s$ is a small copy of $\MM$ within $\MM_s$.

  Every $\ZZ^n_i$ is inside some decoration's
  parapuzzle piece $\ZZ'^m_j$ associated with $\MM'_s$.
  We will show that
  $\AA^n_i:=\ZZ'^m_j\backslash \ZZ^n_i$ are annuli that satisfy the
above requirements%
\footnote{Actually, we will denote by $\AA^n_i$ subannuli of
$Z'^m_j\backslash \ZZ^n_i$ and prove that these satisfy the
conditions stated above; this also would imply that
$\ZZ'^m_j\backslash \ZZ^n_i$ satisfy the above conditions.}.

  By Lemma~\ref{lem1} every $\ZZ^n_i$ corresponds to
   dynamical puzzle pieces $Z^n_i$.
   All puzzle pieces $Z^n_i$ are preimages of a single puzzle piece
 $Z^0_0$, which itself is inside a puzzle piece $\widehat{Z^0_0}$ (Proposition~\ref{defprop1}).

In Section~\ref{sec2.3} we will consider a particular case that we
call ``simple''. We pull back conformally the annulus
$\widehat{Z^0_0}\setminus Z^0_0$ to the annulus
$\widehat{Z^n_i}\setminus Z^n_i$ around $Z^n_i$. For
$\widehat{Z^n_i}\setminus Z^n_i$ there exists the corresponding
annulus $\widehat{\ZZ^n_i}\setminus \ZZ^n_i$ in the parameter plane
with comparable modulus.

    Lemma~\ref{lem4} (Section~\ref{sec2.4}) shows the existence of a
    big collection of annuli with  bounded below modulus. Pulling
    back
    conformally these annuli we obtain a collection of annuli $\widetilde{A}$ within
    $Z'^m_j$, where every annulus has a corresponding annulus in parameter space.
     If the decoration is ``unsimple'' and sufficiently close to $c_0$ (Section~\ref{Sub:Unsimple}), then $Z^n_i$ at $c=c_0$ is
    surrounded by an annulus $A^n_i\in \widetilde{A}$.

\medskip\noindent\textbf{Acknowledgements.} I am very grateful to
 Carsten Lunde Petersen, Mikhail Lyubich, Pascale
Roesch, Davoud Cheraghi and ``Bremen dynamical group'', in
particular Dierk Schleicher, Vladlen Timorin, Nikita Selinger,
Yauhen Mikulich for very useful discussions.

I am very grateful to Dierk Schleicher for his invaluable assistance
in writing this paper.

\newpage

\section{Puzzle and parapuzzle pieces}
\label{sec1}

Let $\ovl R$ be a finite collection of periodic and preperiodic
rays such that:
\begin{itemize}
\item every ray lands;

\item every landing point is the landing point of at least two
rays in $\ovl R$;

\item $\ovl{R}$ is forward invariant: $f_c(R)$ is in $\ovl
R$ for every $R\in \ovl R$.
\end{itemize}

The family $\ovl R$ together with an equipotential gives a
partition of a neighborhood of the Julia set (we assume that all sets in
this partition are closed; their boundaries may intersect).
Any bounded
component $X^0_i$ of the partition is called a puzzle piece of depth
$0$. We say that $X^n_i$ is \textit{a puzzle piece of depth $n$} if
it is a preimage of a puzzle piece of depth $0$ under $f^n_c$.

Denote by $F_{\ovl R}$ the family of all puzzle pieces
associated with $\ovl R$ (of all depths).
\begin{property}
All puzzle pieces from $F_{\ovl R}$ are closed topological discs. If
$X^n_i, X^m_j\in F_{\ovl R}$, then either their interiors do not
intersect or $X^n_i\subseteq X^m_j$; the latter can happen only if
$n\ge m$.
\end{property}

   The following criterion is useful in order
   to determine when a topological disc is a puzzle piece (in an
appropriate family).

\begin{proposition}
\label{prop0.1}
   Suppose a closed topological disc $X$ is bounded by periodic and preperiodic
   rays and truncated by an equipotential. Then there exists a family
$F$ so that $X$ is a puzzle piece for $F$
   if and only if the forward orbit of $\partial X$ does not
   intersect the interior of $X$.
\end{proposition}

Indeed, assume that $X$ is bounded by $R_1, \ldots, R_s$ and
consider $\ovl R=\bigcup_{k\ge 0}\bigcup_j f^k(R_j)$. If the
assumption in Proposition~\ref{prop0.1} is satisfied, then $X$ is a
puzzle piece in the family $F=F_{\ovl R}$. The converse is obvious.

Assume that the critical value $c$ is in the interior of a puzzle
piece $X^0_1$ of depth $0$. It is well known that there exists a
topological disc $\XX^0_1$ in the parameter plane such that the
boundary $\partial \XX^0_1$ has the same combinatorial structure as
$\partial X^0_1$. In addition, for every parameter $c\in \intr
\XX^0_1$, all puzzle pieces from $F_{\ovl R}$ of depth $0$ ``exist''
and depend continuously on $c$. We will say that the family $F_{\ovl
R}$ \textit{exists} in $\XX^0_1$. Depending on properties of $\ovl
R$, the family $F_{\ovl R}$ may exist in a bigger domain.

In general, the family $F_{\ovl R}$ depends on $c$. Let
$X^n_i\in F_{\ovl R}$ be a puzzle piece in the dynamical plane
of $z^2+c_1$, where $c_1\in \intr \XX^0_1$. We will say that
$X^n_i$ \textit{exists} for $c_2\in \intr \XX^0_1$ if there
exists a puzzle piece $\widetilde{X^n_i}$ in the dynamical plane of
$z^2+c_2$ such that the boundaries of $\widetilde{X^n_i}$ and
$X^n_i$ are combinatorially equivalent. To simplify, we will write
$\widetilde{X^n_i}=X^n_i$; this convention allows us to use the
notion of puzzle pieces without referring to a particular dynamical
plane.

Let $\XX^n_i\subseteq \XX^0_1$ be a topological disc
in the parameter plane bounded by parameter rays and an
equipotential. Then $\XX^n_i$ is called the \textit{parapuzzle
piece associated to $X^n_i$} for the family $\mathcal F_{\ovl R}$ if
there exists
a puzzle piece $X^n_i\in F_{\ovl R}$ with the same combinatorics and
the following equality holds:
\begin{equation}
\intr \XX^n_i \ =\ \{c\in\C\mid X^n_i \text{ exists for }c\text{ and
} c\in \intr X^n_i\}.
\end{equation}

The next property shows a relation between puzzle pieces in
different families.

\begin{property}
\label{prop0.2} Suppose a puzzle piece $X^0_i\in F_1$ is inside a
puzzle piece $Y^0_j\in F_2$ of depth $0$. Then any preimage of
$X^0_i$ under $f_c^n$ is inside a puzzle piece of depth $n$ from
$F_2$.

\end{property}

\section{The Combinatorial Construction}
\label{CombinSect}

Let $\MM_{s}$ be a small copy of the Mandelbrot set. Then
the component of $\MM \backslash \MM_{s}$ containing
the main cardioid is ``big,'' and any other component is a part of
$\MM$ cut off by two external rays landing at some tip of
$\MM_{s}$ (such a component is called a \textit{decoration}
\cite{KL}; and a tip of $\MM_{s}$ must be a Misiurewicz
point).

To be more precise, each $\MM_s$ has an integer $q\ge 2$ so that
each tip of $\MM_s$ is the landing point of exactly $q$ parameter
rays. They chop off $q - 1$ decorations $\LL_{k}$ (closures of the
components of $\MM\ \backslash \{t\}$ that do not intersect the main
cardioid) from $\MM$.

If the decoration conjecture was not true, then there
   would be some $\eps>0$ and infinitely many decorations
   $\LL_{1},\LL_{2},\ldots$ with diameters at least
   $2\eps$. Denote by $a_i$ the tip of $\LL_i$ (defined as
the point of intersection of $\LL_i$ and $\MM_s$).
We may extract a subsequence so that all $a_i$ are different.
   Now let us choose in each $\LL_{i}$ any point $b_{i}$
   such that $|a_{i}-b_{i}|\ge \eps$; let $c_{0}$
   be an accumulation point of the sequence $\{b_{i}\}$; we may assume
that all $|b_i-c_0|<\eps/2$.

   \begin{proposition}
The quadratic map $f_{c_{0}}(z)=z^{2}+c_{0}$ is an infinitely
   renormalizable polynomial. Moreover, $c_{0}$ belongs to $\MM_{s}$.
\end{proposition}

\begin{proof}
    Every decoration $\LL_{i}$ is separated from $\MM\sm\LL_{i}$
     by two parameter
    rays landing at $a_{i}\in \MM_s$. In addition, $\LL_{j}\subset
    \MM\sm\LL_{i}$ for $i\not=j$;
    therefore the accumulation point $c_0$ of decorations can not be
in any $\LL_{i}$ unless
    $c_0=a_i$. Thus $c_0\in\MM_s$.

    Since $|a_i-b_i|\ge\eps$ but $|b_i-c_0|<\eps/2$, the sets
\[
N_{\eps/2}(c_0)\cap \MM_s \quad \text{and} \quad N_{\eps/2}(c_0)\cap L_i
\]
are all disjoint (where $N_{\eps/2}(c_0)$ denotes the
$\eps/2$-neighborhood of $c_0$). This implies that $\MM$ is not
locally connected at $c_0$. By Yoccoz's results (see \cite{Hu})
 $z^{2}+c_{0}$ is an infinitely
   renormalizable polynomial.
\end{proof}

\subsection{Small copies of the Mandelbrot set}

\label{sec2.1}

   A repelling periodic cycle
$\ovl{\alpha}=\{\alpha_k\}_{k=0}^{n-1}$ is called \textit{dividing}
if there are at least two rays landing at each $\alpha_k$. By
$R=R(\ovl{\alpha})$ we denote the configuration of rays landing at
$\ovl{\alpha}$.

\begin{figure}
\includegraphics[width=9cm]{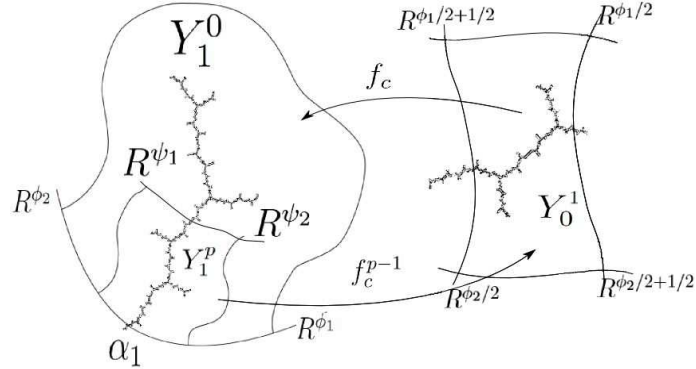}
  \caption{ Construction of a renormalization domain.}
\end{figure}

\begin{proposition}[{see \cite{Mi}}]
\label{prop1}
   Let $\ovl{\alpha}=\{\alpha_k\}_{k=0}^{n-1}$ be a dividing repelling
periodic cycle.
\begin{itemize}
\item  Let $Y_1$
be the component of $\mathbb{C}\backslash R(\ovl\alpha)$ containing
the critical value $c$. Then $Y_1$ is a sector bounded by two
external rays.
\item Let $Y_0$ be the component of $\mathbb{C}\backslash
f_c^{-1}R(\ovl\alpha)$
  containing the critical
point $0$. Then $Y_0$ is bounded by four external rays: two of them
land at a periodic point $\alpha_k$, and two others land at the
symmetric point $-\alpha_k$.
\item The rays of $R(\ovl\alpha)$ form either one or two cycles under
iterates of $f_c$. All cycles have the same period.

\end{itemize}

\end{proposition}

Let $p\ge 2$ be the period of the rays $R(\ovl\alpha)$, let
$\alpha_0$ be the periodic point on the boundary of $Y_0$, and let
$\alpha_1$ be the point on the boundary of $Y_1$. Define $Y^0_1$ to
be $Y_1$ truncated by the equipotential of height $1$. Let $Y^p_1$
be the unique component of $f_c^{-p}(Y^0_1)$ attached to $\alpha_1$.
By
  construction we have $Y^p_1\subsetneqq Y^0_1$.

The following definition is one of several equivalent ways of
defining (simple) renormalization and small Mandelbrot sets:

\begin{definition}
\label{defMM} A quadratic map $f_c$ is called \textbf{DH
renormalizable} of period $p$ if there exists a cycle $\ovl\alpha$
as above such that $c$ will not escape from $Y^p_1$ under iteration
of $f^p:Y^p_1 \rightarrow Y^0_1$ (in particular $c\in Y^p_1$). The
associated \textbf{small copy of the Mandelbrot set} is the closure
of the set of parameters $c$ such that $z^2+c$ is DH renormalizable
with the fixed ray pattern $R(\ovl\alpha)$.
\end{definition}

It is known \cite{DH} that a small copy of the Mandelbrot set is
indeed canonically homeomorphic to $\MM$.

\textit{For the rest of the paper, we will fix the cycle
$\ovl{\alpha}$ and the corresponding small copy $\MM_s$ of the
Mandelbrot set.}

   We now give a detailed description of the construction of $Y^p_1$.
   By Proposition~\ref{prop1} the sector $Y_1$ is bounded by two
   rays;
   denote them by $R^{\phi_1}$ and $R^{\phi_2}$.
   The strip $Y_0$ is bounded by two pairs of rays. One of them lands
at the periodic point $\alpha_0$; we can assume that this ray pair
has the angles $R^{\phi_1/2+1/2}$, $R^{\phi_2/2}$ (possibly by
interchanging $\phi_1$ and $\phi_2$). The other pair is then
$R^{\phi_1/2}$, $R^{\phi_2/2+1/2}$ and lands at $-\alpha_0$.
   The strip $\lfloor Y^p_1 \rfloor$ is defined as $f^{-p+1}(Y_0)$,
   where the pullback is taken along the
   orbit of the periodic rays $R^{\phi_1}$ and $R^{\phi_2}$.
   Further, the strip $\lfloor Y^p_1 \rfloor$ is bounded by two pairs of rays
   and one of them consists of $R^{\phi_1}$, $R^{\phi_2}$ landing at
   $\alpha_1$; denote by $R^{\psi_1}$ and $R^{\psi_2}$ the other
   pair. The last two rays depend continuously on the parameter $c$ whenever
   $R^{\phi_1}$ and $R^{\phi_2}$ do (the critical value can not cross
the forward orbit of $R^{\psi_1}$ and
   $R^{\psi_2}$).

It is known \cite{Mi} that the parameter rays $\RR^{\phi_1}$ and
$\RR^{\phi_2}$ land together at the root of $\MM_s$. We define
$\YY_1$ to be the sector bounded by
$\ovl{\RR^{\phi_1}\bigcup\RR^{\phi_2}}$ and containing $\MM_s$.
Whenever $c\in \intr \YY_1$ the ray configuration $R(\ovl{\alpha})$
depends holomorphically on $c$. The next statement implies that
``small Julia sets do not intersect'' (except at points of the orbit
$\ovl\alpha$); the proof follows from Proposition~\ref{prop1} and
the definition of the strip $Y^p_1$.

\begin{property}
\label{property1}
  If $c$ belongs to the sector $\YY_1$, then for all $0\le k<p$
  the interiors of strips $f_c^k \left(\lfloor Y^{p}_1\rfloor\right)$
  are disjoint and do not intersect $Z^0_0$.
\end{property}

Let the parapuzzle piece $\YY_1^0$ be the sector $\YY_1$ truncated
by the equipotential of height $1$. If $c$ belongs to $\intr
\YY_1^0$, then the puzzle pieces $Y^p_1$ and $Y^0_1$ depend
continuously on $c$. \textit{From now we assume that $c\in \intr
\YY_1^0$}.

\begin{figure}
\includegraphics[width=6cm]{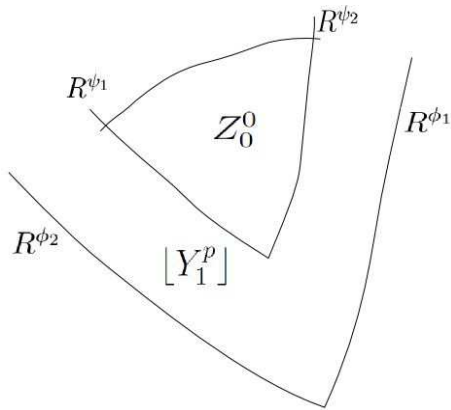}
  \caption{The puzzle piece $Z^0_0$ and the strip $\lfloor
  Y^p_1\rfloor$.
  \label{Fig:ZZ}
  }

\end{figure}
By $Z^0_0$ we denote the sector bounded by the rays $R^{\psi_1}$ and
    $R^{\psi_2}$, not containing $0$, and truncated by the equipotential
    of height $1/2$ (Figure~\ref{Fig:ZZ}).

Let $\ovl R=\bigcup_{k\ge0} \{f^k_c(R^{\psi_1}),f^k_c(R^{\psi_2})\}$
be the forward orbit of the rays $R^{\psi_1}$ and
    $R^{\psi_2}$; it contains $R(\ovl\alpha)$ but perhaps not all the rays that land at the same point as
$R^{\psi_1},R^{\psi_2}$. Using $\ovl R$ and the equipotential of
height
    $1/2$, we get a partition of a neighborhood of the
    Julia set. Denote by $F$ the corresponding puzzle family. From
    the construction and Proposition~\ref{prop1}, it follows that $F$
exists in $\YY^0_1$
     and the
    puzzle piece $Z^0_0$ is in $F$.
(Note that with respect to this family, $Y^0_1$ is not a puzzle
piece but the union of two such pieces, separated by the rays
$R^{\psi_1},R^{\psi_2}$. The set $Y^p_1$ is not a puzzle piece in
$F$ either because it is bounded by the equipotential at height
$2^{-p}$; but there is a puzzle piece of depth $0$ bounded by the
same rays as $Y^p_1$ and the equipotential at height $1/2$.)

\begin{definition}
\label{def1} We define the collection $\{Z^{n}_{i}\}\subset F$ as
the set of maximal conformal pullbacks of $Z^0_0$. This means that
$Z^{n}_{i}\not\subseteq Z^{m}_{l}$ and $f^{n}$ maps conformally
$Z^{n}_{i}$ to $Z^0_0$ for any $i,l$ and $m<n$.
\end{definition}

  Let $\LL_{i,1},\ldots, \LL_{i,q-1}$ be the
group of $q-1$ decorations touching a common Misiurewicz point $a_i$
which is a tip of $\MM_s$. There exists a pair of parameter rays
$\RR_1$, $\RR_2$ that land together at $a_i$ and separate the group
$\bigcup_{j=1}^{q-1} \LL_{i,j}$ from $\MM\setminus
\bigcup_{j=1}^{q-1} \LL_{i,j}$. Also, if $c$ belongs to
$\bigcup_{j=1}^{q-1} \LL_{i,j}$, then there
  exists a constant $n=n(i)$ such that $f_{c}^{n}(0)\in
  Z^{0}_{0}$ but $f_{c}^{k}(0)\not\in
  Z^{0}_{0}$ for $0\le k<n$ (the constant $n(i)$ is called the
\emph{escaping time} \cite{L1}).

\begin{definition}
\label{def2} Denote by $\ZZ^{n}_{i}$ the parapuzzle piece containing
$\bigcup_{i=j}^{q-1} \LL_{i,j}$,
  bounded by $\RR_1$ and $\RR_2$, and truncated by
  the equipotential of height $1/2^n$.
\end{definition}

The parapuzzle pieces $\ZZ^{n}_{i}$ correspond to dynamical puzzle
pieces $Z^{n}_{i}$:

\begin{lemma}[Parapuzzles Pieces]
\label{lem1}  For any $n$ the set
\[
\JJ_n:=\intr\YY^{0}_{1}\backslash \bigcup_{m\le n-1}\bigcup_j
\lfloor\ZZ^{m}_{j}\rfloor
\]
is an open Jordan disk. When $c$ moves in $\JJ_n$, the boundaries of
all $Z^{n}_{i}$
  are disjoint and move holomorphically; moreover
there is a one to one correspondence between $Z^{n}_{i}\subset
Y^{0}_{1}$ and $\ZZ^{n}_{i}\subset\YY^{0}_{1}$ such that:
$$c\in \ZZ^{n}_{i} \text{ if and only if } c\in
Z^{n}_{i}.$$
\end{lemma}

\begin{proof}
This is a classical result (see for example \cite[Lemma $3.3$,
10]{ALS}).
\end{proof}

\subsection{Secondary Decorations}

\label{sec2.2}

As $z^2+c_0$ is infinitely renormalizable, there exists a small copy
$\MMp_s$ of the Mandelbrot set such that $\MMp_s$ contains $c_0$ and
$\MMp_s\subsetneq\MM_s$. Let $\ovl{\alpha'}$ be the dividing
periodic cycle associated with $\MMp_s$, and let $p'$ be the period
of the rays. In general, \textit{we will use an apostrophe to mark
that some structure is associated with $\MMp_s$} (except in
Subsection~\ref{Sub:Unsimple}). All above results for $\MM_s$ are
true for $\MM'_s$. For instance, the puzzle and parapuzzle pieces
$\YY'^0_1$, $Z'^m_j$, $\ZZ'^m_j$ are defined in the same way as
$\YY^0_1$, $Z^n_i$, $\ZZ^n_i$.

More precisely, the sector $Y'_1$ is bounded by a periodic dynamic
ray pair $R^{\phi'_1}$ and $R^{\phi'_2}$. Pulling back this sector
for $p'$ iterations, we obtain a strip within $Y'_1$ that is bounded
by two dynamic ray pairs consisting of $R^{\phi'_1}$ and
$R^{\phi'_2}$ respectively of $R^{\psi'_1}$ and $R^{\psi'_2}$. The
latter two rays, together with their forward orbits, define the
family $F'$ of puzzle pieces\footnote{As before, $Y'^0_1$ is not in
$F'$.} associated to $\MMp_s$. In particular, $Z'^0_0\in F'$ is the
puzzle piece bounded by the two rays $R^{\psi'_1}$ and $R^{\psi'_2}$
and the equipotential at height $1$. The parapuzzle piece $\YYp^0_1$
is bounded by the parameter ray pair $\RR^{\phi'_1}, \RR^{\phi'_2}$
and truncated at height $1$. The family $F'$ exists for all $c\in
\YYp^0_1$. The puzzle pieces $Z'^m_j$ are maximal conformal
pull-backs of $Z'^0_0$, and if $Z'^m_j\subset Y'^0_1$, then
$\ZZ'^m_j$ exists.

\begin{figure}
\includegraphics[width=0.45\textwidth]{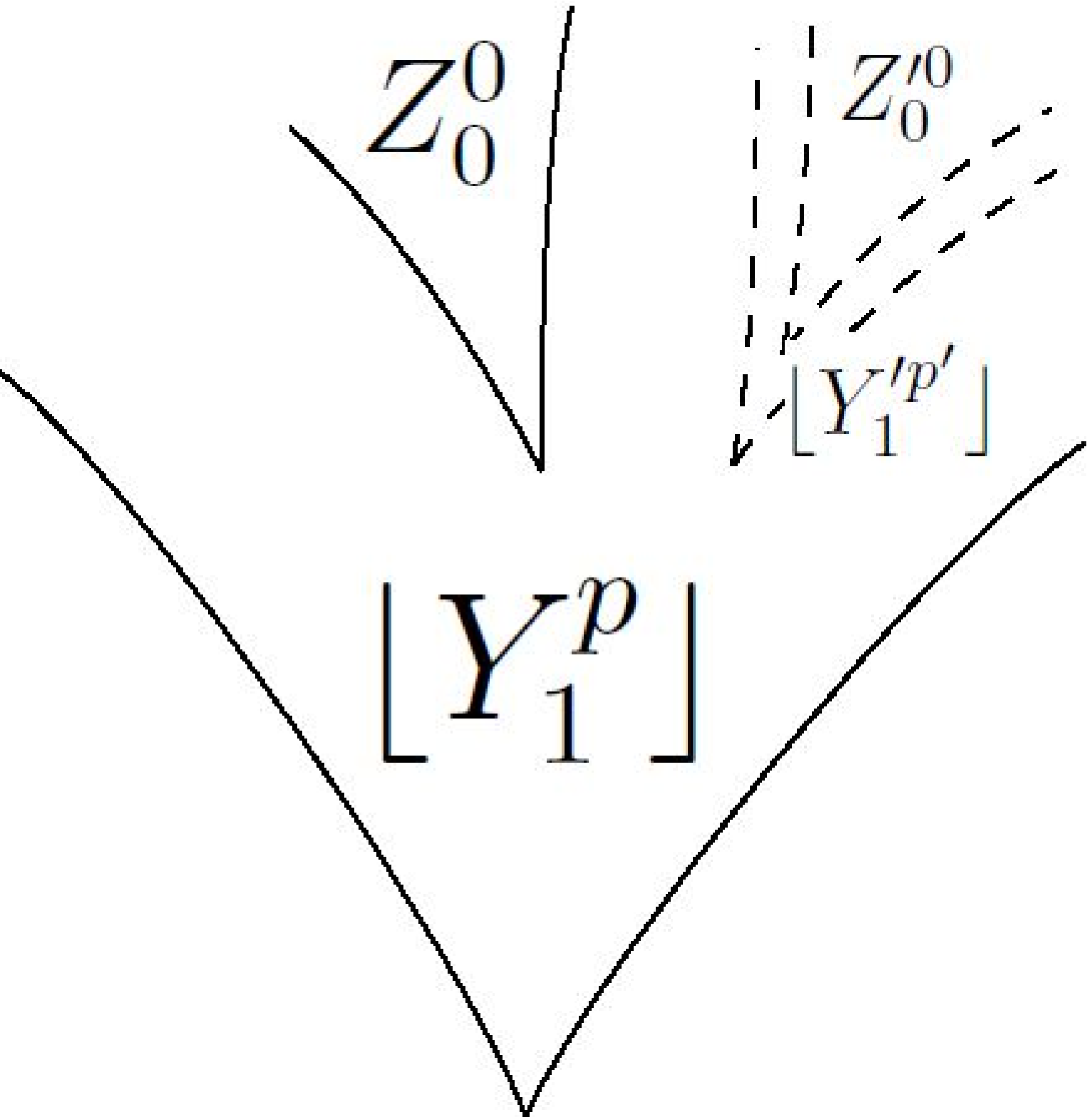}
\includegraphics[width=0.45\textwidth]{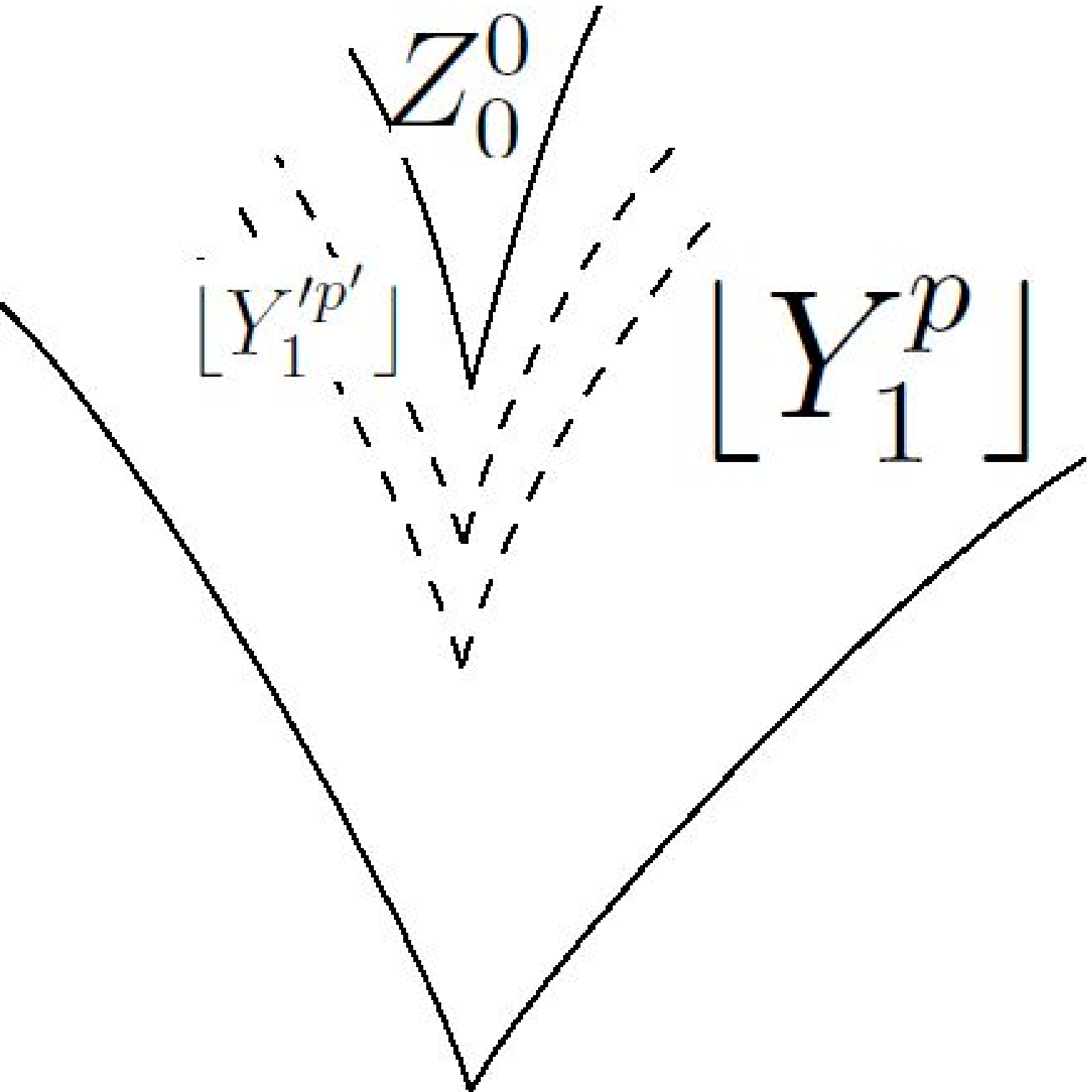}
  \caption{Two possibilities for the strip $\lfloor
Y'^{p'}_1\rfloor$. Right: $\lfloor Y'^{p'}_1\rfloor$ separates
$Z^0_0$ and $0$. Left: $\lfloor Y'^{p'}_1\rfloor$ does not separate
$Z^0_0$ and $0$.} \label{Fig:Strips}
\end{figure}

  Without loss of generality we can assume that
all decorations $\LL_i$ are inside $\YYp^0_1$. \textit{From now on
we will only consider parameters $c\in \intr \YYp_1^0$} (observe
that $\YYp^0_1\subset\YY^0_1$). The next proposition is
``equivalent'' to $\MMp_s\subset \MM_s$.

\begin{proposition}
\label{Prop:SubStrip}
 For every $k'<p'$ there exists a $k<p$ such that
 \begin{equation}
 \label{StripeIncl}
 f^{k'}_c\left(\lfloor Y'^{p'}_1\rfloor
 \right)\subset \intr  f^{k}_c\left(\lfloor Y^{p}_1\rfloor
 \right).
 \end{equation}

\end{proposition}
\begin{proof}
Let $x_1$ and $x_2$ be the two landing points of the two ray pairs
bounding $\lfloor Y'^{p'}_1 \rfloor$. These belong to the ``small
Julia set'' associated with $\MM_s$ (for the parameter $c_0$), hence
$\{x_1,x_2\}\subset Y^p_1$.

Then for any $k'$ there exists a $k\in\{0,1,\dots,p-1\}$ (in fact,
$k\equiv k' \pmod p$) such that $f^{k'}_c\left(\{x_1,x_2\}
 \right)\subset \intr  f^{k}_c\left(\lfloor Y^{p}_1\rfloor
 \right)$.

We need to show that $f^{k'}_c\left(\lfloor Y'^{p'}_1\rfloor
 \right)$ does not intersect
 $f^{k}_c\left(\partial\lfloor Y^{p}_1\rfloor
 \right)$ if $k'<p'$.
Recall that we have the following proper surjective maps (of degree
$2$):
$$f_c^{p}:\ \lfloor Y^{p}_1 \rfloor\rightarrow Y_1,$$
$$f_c^{p'}:\ \lfloor Y'^{p'}_1 \rfloor\rightarrow Y'_1,$$
where $Y'_1\subset\intr Y_1$. Therefore if the above statement is
not true, then $Y'_1$ would have non trivial intersection with
$\partial Y_1$.
\end{proof}

We will refer to decorations associated to $\MM_s$ as \emph{primary
decorations},
  and to those associated to $\MM'_s$
as \emph{secondary decorations} (both in dynamical and in parameter
spaces),
  and similarly for puzzle pieces.
In addition, we need another collection of puzzle pieces associated
with $\MM'_s$ that will be denoted $\widehat{Z^n_i}$.

\begin{defprop}[The Fundamental Annulus]
\label{defprop1}
   For every $c\in\YY'^1_0$, there exists a puzzle piece $\widehat{Z^0_0}\in F'$ of depth $0$
   the interior of which contains $Z^0_0$.
For $k'<p'$ we have $\intr\widehat{Z^0_0}\cap f^{k'}_c(\lfloor
Y'^{p'}_1\rfloor) = \emptyset$. The boundary of the annulus
$\widehat{Z^0_0}\setminus
   Z^0_0$ depends holomorphically on $c$.

\end{defprop}

Note that $\widehat{Z^0_0}$ is inside $Y_1$ but may or may not be
inside $Y'_1$. In fact, $\widehat{Z^0_0}$ is inside $Y'_1$ if and
only if $\widehat{Z^0_0}=Z'^0_0$ (see Figure~\ref{Fig:Strips}).

\begin{proof}
We need to show that for any puzzle piece $X'\in F'$ of depth $0$,
we have $\partial X'\cap Z^0_0=\emptyset$. But any $X'\in F'$ at
depth $0$ is truncated by
  the equipotential of greater height than $Z^0_0$ (height $1$ for
$F'$ and height $1/2$ for $F$). Therefore it is enough to verify
that
  if a ray $R'$ bounds $X'$,  then $R'$ does not intersect
  $Z^0_0$. Note that $R'$ belongs to $\partial f^{k'}_c\left(\lfloor Y'^{p'}_1\rfloor
 \right)$ for some $k'< p'$, where $f^{k'}_c\left(\lfloor Y'^{p'}_1\rfloor
 \right)\subset \intr f^{k}_c\left(\lfloor Y^{p}_1\rfloor
 \right)$ for some $k<p$ (see Proposition~\ref{Prop:SubStrip}). Hence $R'$ does not
 intersect $Z^0_0$ (see Property~\ref{property1}); this proves the existence
 of $\widehat{Z^0_0}$. Holomorphic dependence of the boundary is clear as long as
 all pieces exist, which is the case for all $c\in \YYp^0_1$.

For $k'<p'$ the sets $f^{k'}_c(\lfloor Y'^{p'}_1\rfloor)$ are
untruncated puzzle pieces in $F'$ of depth $0$ by construction, and
so is $\lfloor \widehat{Z^0_0}\rfloor$. But $f^{k'}_c(\lfloor
Y'^{p'}_1\rfloor)$ do not intersect $Z^0_0$ as we just showed, so
these puzzle pieces have disjoint interiors.
\end{proof}

By definition, any $Z^n_i$ is a conformal pullback of $Z^0_0$ under
a branch of $(f^n_c)^{-1}$. It may or may not be possible to
conformally pull back the larger domain  $\widehat{Z^0_0}$ along the
same orbit (i.e., choosing the same preimage branches). If it is, we
denote the puzzle piece thus obtained by $\widehat{Z^n_i}\supset
Z^n_i$. It is clear that $\widehat{Z^n_i}\in F'$.

The following proposition illustrates the relation of puzzle and
parapuzzle pieces associated with $\MM_s$ and $\MMp_s$:

\begin{proposition}
\label{prop1.0}
  For any $Z^n_i$ and any $c\in \intr \YYp^0_1$, the following properties hold:
\begin{enumerate}
\item
if $Z^n_i$ and $\intr Y'^{p'}_1$ intersect, then $Z^n_i\subset \intr
Y'^{p'}_1$;
\item
if $Z^n_i$ and $Z'^m_j$ intersect with $m\le n$, then $Z^n_i\subset
\intr Z'^m_j$; in addition, if $\widehat{Z^n_i}$ exists, then
$\widehat{Z^n_i}\subset Z'^m_j$;
\item
if $Z^n_i$ and $Z'^m_j$ intersect within $Y'^{p'}_1$, then $m\le n$
and $Z^n_i\subset \intr Z'^m_j$;
\item
if $c=c_0$, then $\widehat{Z^n_i}$ exists and is contained in some
$Z'^m_j$ with $m\le n$.
\end{enumerate}

\end{proposition}

\begin{proof}
 The set $Z^0_0$ is contained in a puzzle piece $\widehat{Z^0_0}$ of $F'$ of
depth $0$ (Proposition~\ref{defprop1}), so by induction $Z^n_j$ is
contained in a puzzle piece (say $X'^n$) of $F'$ of depth $n$. Hence
if it intersects a puzzle piece $Z'^m_j$ of depth $m\le n$, then it
is contained in that piece. If $\widehat{Z^n_i}$ exists, then it
equals $X'^n$ by construction. This proves (2).

Similarly, if $Z^n_i$ and $\intr Y'^{p'}_1$ intersect, then
$Z^n_i\subset \intr X'^n\subset \intr Y'^{p'}_1$; this is (1).

Let us now show that if $Z^n_i$ and $Z'^m_j$ intersect within
$Y'^{p'}_1$, then $m/p'$ is the ``escaping time'' for $Z^n_i$ under
$f_c^{p'}:\lfloor Y'^{p'}_1\rfloor \rightarrow Y'_1$, i.e., $m/p'$
is the least iterate so $(f^{p'}_c)^{m'/p'}(Z^n_i)\not\subset
\lfloor Y'^{p'}_1 \rfloor$; this will imply that $m\le n$.

Observe that, up to truncation, the two domains $Y'^{p'}_1$,
$Z'^0_0$ are puzzle pieces in $F'$ of depth $0$. By
Proposition~\ref{defprop1} we have $n\ge p'$, so the map $f_c^{p'}$
sends $X'^n$ to a puzzle piece (say $X'^{n-p'}$) in $Y'_1$. Either
$X'^{n-p'}$ is a subset of $Z'^0_0$ (in this case $m=p'$) or
$X'^{n-p'}$ is a subset of $Y'^{p'}_1$; in the second case we use
induction. This proves (3) (using (2)).

Assume that $c=c_0$ (so that $f_{c_0}$ is renormalizable with
respect to $\MMp_s$); thus all $Z'^m_j$ exist. Therefore if $Z^n_i$
is in $Y'^{p'}_1$, then $Z^n_i$ is in some $Z'^m_j$, where
$c_0\not\in Z'^m_j$. By induction on $n$ it is easy to prove that
$\widehat{Z^n_i}$ exists and has conformal pullbacks.
\end{proof}

If $c=c_0$, then for every $Z^n_i\subset Y'^0_1$ there is a $Z'^m_j$
so that $Z^n_i\subset\widehat{Z^n_i}\subset Z'^m_j$ (the depth $m$
is unique because of the condition that $Z'^m_j$ be maximal).
\textit{Let us fix $Z^n_i\subset Y'^0_1$ and the corresponding
$Z'^m_j$.}

We will distinguish the following two cases (depending on $c_0$):
the decoration ${Z^n_i}$ is
\begin{description}
\item[simple] if
$f_{c_0}^k\left(\lfloor \widehat{Z^n_i}\rfloor\right)\not\subset
\lfloor Z'^m_j\rfloor$ for all $k\in\{1,2,\dots,n\}$;

\item[unsimple] otherwise; i.e., if there is a $k\in\{1,2,\dots,n\}$ with
$f_{c_0}^k\left(\lfloor \widehat{Z^n_i}\rfloor\right)\subset \lfloor
Z'^m_j\rfloor$.
\end{description}

In the simple case, the proof will be based on the fact that
$\widehat{\ZZ^n_i}$ exists, which will provide a fundamental annulus
of uniform modulus. In the unsimple case, $Z^n_i$ is surrounded by
an annulus that is a conformal pullback of an annulus in
Lemma~\ref{lem4}.

\section{The Proof}

  \subsection{The Simple Case}

\label{sec2.3}

 The purpose of this
  section is to prove the following statement:

\begin{proposition}[Fundamental Annulus]
\label{prop2} Let $Z^n_i$ be puzzle piece of a simple decoration.
Then $\widehat{\ZZ^n_i}\subset \ZZ'^m_j\subset\intr \YYp^0_1$
exists, the annulus $\widehat{Z^n_i}\backslash Z^n_i$ exists and
moves holomorphically (with the base point $c_0$) in some open
Jordan disc $\JJ\subset \YY'^0_1$ containing $\ZZ'^m_j$, and the
quasiconformal
  dilatation of the holomorphic motion for any parameter $c\in \JJ$ is bounded in
  terms of the conformal distance of $c$ to the boundary of $\YY'^0_1$.
  \end{proposition}

  The main point of this proposition is that the quasiconformal
  dilatation of the motion is bounded in terms of the distance
  to the boundary of $\YY'^0_1$ (not $\JJ$).

  \begin{proof}
  From the construction we see that the boundary of the annulus
  $\widehat{Z^0_0}\setminus
  Z^0_0$ moves holomorphically whenever $c\in \intr\YY'^0_1$ (Proposition~\ref{defprop1}). By the
  $\lambda-$lemma we can extend this motion to a holomorphic motion $h_c$ of the closure of the annulus itself,
  and so that the quasiconformal dilatation is bounded above in terms of the
distance to the boundary of
   $\YY'^0_1$. Whenever $\tilde c\not\in
   f_{\tilde c}^k\left(\lfloor\widehat{Z^n_i}\rfloor\right)$ for all
$k=1,\ldots n$, then $\tilde c$
   has a neighborhood so that for all $c$ from this neighborhood we
can pull back
   the holomorphic motion $h_c$ along the orbit
$\widehat{Z^n_i},f_c\left(\widehat{Z^n_i}\right),\ldots ,
   f_c^n\left(\widehat{Z^n_i}\right)=\widehat{Z^0_0}$ by the formula
$h^n_c=f_c^{-n}\circ h_c\circ
   f_{\tilde c}^n$. Such a pull-back does not change the quasiconformal
   dilatation.

Define $T:=\left\{k\in\{1,\dots,n\}\colon
f_{c_0}^k\left(\widehat{Z^n_i}\right)\subset Y'^0_1\right\}$. For
each $k\in T$, there exist $m_k\le n-k$ and $j_k$ so that
$f_{c_0}^k\left(\widehat{Z^n_i}\right)=\widehat{Z^{n-k}_{i'}}\subset
Z'^{m_k}_{j_k}$ for some $i'$ (Proposition~\ref{prop1.0} (4)).

Recall from Lemma~\ref{lem1} that all parapuzzle pieces $\ZZ^n_i$
exist and, similarly, all $\ZZ'^m_j$ exist and are contained in
$\intr\YYp^1_0$. Define
\[
\JJ:=\intr \YY^0_1\setminus \bigcup_{k\in T}\lfloor
   \ZZ'^{m_k}_{j_k} \rfloor.
\]
By the definitions of $(m_k,j_k)$ and of simplicity, we have
$(m,j)\neq(m_k,j_k)$ for all $k\in T$; therefore, $\ZZ'^m_j\subset
\JJ$ (because all secondary decoration parapuzzle pieces
$\ZZ'^{m'}_{j'}$are disjoint).

Since for the parameter $c_0$, we have
$Z^n_i\subset\widehat{Z^n_i}\subset Z'^m_j$ by
Proposition~\ref{prop1.0} (4), it follows that
 $\ZZ^n_i\subset  \ZZ'^m_j\subset \JJ$.

The set $\JJ$ has been constructed so that for all parameters
$c\in\JJ$, the dynamical puzzle piece $\widehat Z^n_i$ exists and
depends holomorphically on $c$ (those parameter puzzle pieces for
which problems could occur are exactly the ones that are removed in
the definition of $\JJ$). But $Z^n_i\subset \widehat{Z^n_i}$, so
$Z^n_i$ also moves holomorphically, and thus also the annulus
$\widehat{Z^n_i} \setminus   Z^n_i$. Observe that the quasiconformal
dilatation was introduced only for the annulus
$\widehat{Z^0_0}\setminus Z^0_0$; everywhere else a conformal
preimage of this dilatation was used.

By definition, when $c\in \ZZ^n_i$, then the critical value $c$
    must be inside $Z^n_i\subset \widehat{Z^n_i}$; but since
$\ZZ'^m_j\subset\JJ$, there are parameters $c\in\JJ$ where the
critical value is not in $Z'^m_j\supset \widehat {Z^n_i}$.
Therefore, the dynamic rays that land together and form the boundary
of $\lfloor \widehat{Z^n_i} \rfloor$ have counterparts in parameter
space at the same angles, and these form the boundary of
$\lfloor\widehat{\ZZ^n_i}\rfloor$.

Therefore the parapuzzle piece $\widehat{\ZZ^n_i}$ exists and is
compactly contained in $\JJ\subset\YYp^0_1$.
\end{proof}

\begin{cor}[Annuli in Parameter Space, Simple Case]
\label{Cor:ParaAnnuli} The sets $\widehat{\ZZ^n_i}\setminus \ZZ^n_i$
(for all $n$ and $i$ that correspond to simple cases) are
non-degenerate annuli within $\ZZ'^m_j$, and their moduli are
bounded below in terms of their distance to $\partial\YY'^0_1$.
\end{cor}
\begin{proof}
If $Z^n_i$ exists and corresponds to a simple decoration, then
$\widehat{\ZZ^n_i}$ and $\ZZ^n_i$ exist by Proposition~\ref{prop2}.
We have $Z^n_i\subset \widehat{Z^n_i}$ by construction, and their
boundaries are disjoint. Thus $\widehat{\ZZ^n_i}\setminus \ZZ^n_i$
is a non-degenerate annulus. Its modulus is at least as big as the
modulus of $\widehat{Z^n_i}\sm Z^n_i$ for the parameter $c_0$, up to
the quasiconformal distortion between the $\widehat{Z^n_i}\sm Z^n_i$
for the various parameters $c$ (see \cite{ALS} or \cite{L2}), and
this depends only on distance of $\widehat{\ZZ^n_i}$ to the boundary
of $\YYp^0_1$.
\end{proof}

\begin{theorem}[The Decoration Theorem, Simple Case]
Suppose that $\LL_\nu$ is a sequence of decorations of $\MM_s$ and
$c_0$ is a limiting parameter of them. If all these decorations are
simple (with respect to $c_0$), then diameters of $\LL_\nu$ tend to
zero.
\end{theorem}
\begin{proof}
Each decoration $\LL_\nu$ is contained in a primary parapuzzle piece
${\ZZ}^n_i$ (where $n=n(\nu)$ and $i=i(\nu)$), which in turn is
contained in a secondary parapuzzle piece ${\ZZ'}^m_j$, with
$c_0\not\in {\ZZ'}^m_j$. The ${\ZZ'}^m_j$ must accumulate at $c_0$
and thus cannot accumulate at the boundary of the wake of $\MM'_s$,
which is $\YY'^0_1$ ($c_0$ is infinitely renormalizable and thus
separated from $\partial \YY'^0_1$ by infinitely many parameter ray
pairs, and these ray pairs separate ${\ZZ'}^m_j$ from the wake
boundary for sufficiently large $m$).

For sufficiently large $\nu$, the moduli of the annuli
$\widehat{\ZZ^n_i}\setminus \ZZ^n_i$ are bounded uniformly for all
$\nu$ by Corollary~\ref{Cor:ParaAnnuli}. But
$c_0\not\in\ZZ'^m_j\supset \widehat{\ZZ^n_i}$, so the diameters of
the $\ZZ^n_i$ must tend to zero.
\end{proof}

This concludes the proof in the simple case.

  \subsection{Constructing More Annuli}

\label{sec2.4}

In the following lemma $X$ and $V$ may be associated to arbitrary
dividing periodic cycles:

\begin{lemma}
\label{lem4} In the dynamical plane of $z^{2}+c_{0}$ there are two
puzzle pieces $X$ and $V$ with the following properties:
\begin{itemize}
\item  $c_{0}\in V\Subset X\subset Y'^{0}_{1}$;
\item any iterated preimage of $X$ is either inside ${Z'^0_0}$
or has non empty intersection with it;

\item the parapuzzle pieces $\XX$ and $\VV$ corresponding to $X$ and $V$ exist;
\item there exists an $\varepsilon_{0}=\varepsilon_{0}(X,V)>0$ such that
any iterated preimage $V^{n}$ of $V$ which is inside $X$ satisfies:
\[
\mod(X\sm V^{n})\ge \varepsilon_{0}.
\]

\end{itemize}
\end{lemma}

\begin{proof}

Since $c_0$ is infinitely renormalizable, there are infinitely many
nested renormalization domains, and these are contained  in $\lfloor
Y'^{p'}_1\rfloor $ provided the level $N$ of the renormalization is
sufficiently large.

We shall prove that if $X$ and $V$ are two renormalization domains
around $c_0$ of levels $N$ and $N+2$ for sufficiently large $N$ and
truncated at sufficiently small heights, then all four properties
are satisfied. Both domains are bounded by two pairs of dynamic rays
and one equipotential; the landing points of these ray pairs will be
called the vertices of $\partial X$ or $\partial V$. One of the
vertices will be a periodic point, the other one preperiodic on the
same orbit.

Then the first condition is satisfied, and the second follows
because $Z'^0_0$ is bounded by a dynamic ray pair outside of the
secondary renormalization domain, as well as a fixed equipotential
(see Property~\ref{prop0.2} for illustration).

The third claim also follows by standard results.

The last claim is similar to \cite[Lemma~4.5]{L1}. We will give a
sketch of the argument.

Let $V^n$ be an iterated preimage of $V$ with $V^n\subset X$. For
every $\eta>0$ there is a $\delta(\eta)>0$ so that if $V^n$ has
distance at least $\eta$ to both vertices of $\partial X$, then
$V^n$ must have distance at least $\delta(\eta)$ from $\partial X$
(see Figure~\ref{Fig:LyubichLemmaPart1} and its caption). This
implies that $\bmod(X\setminus V^n)\ge \eps_1(X,V)>0$.

\begin{figure}
\includegraphics[width=0.4\textwidth]{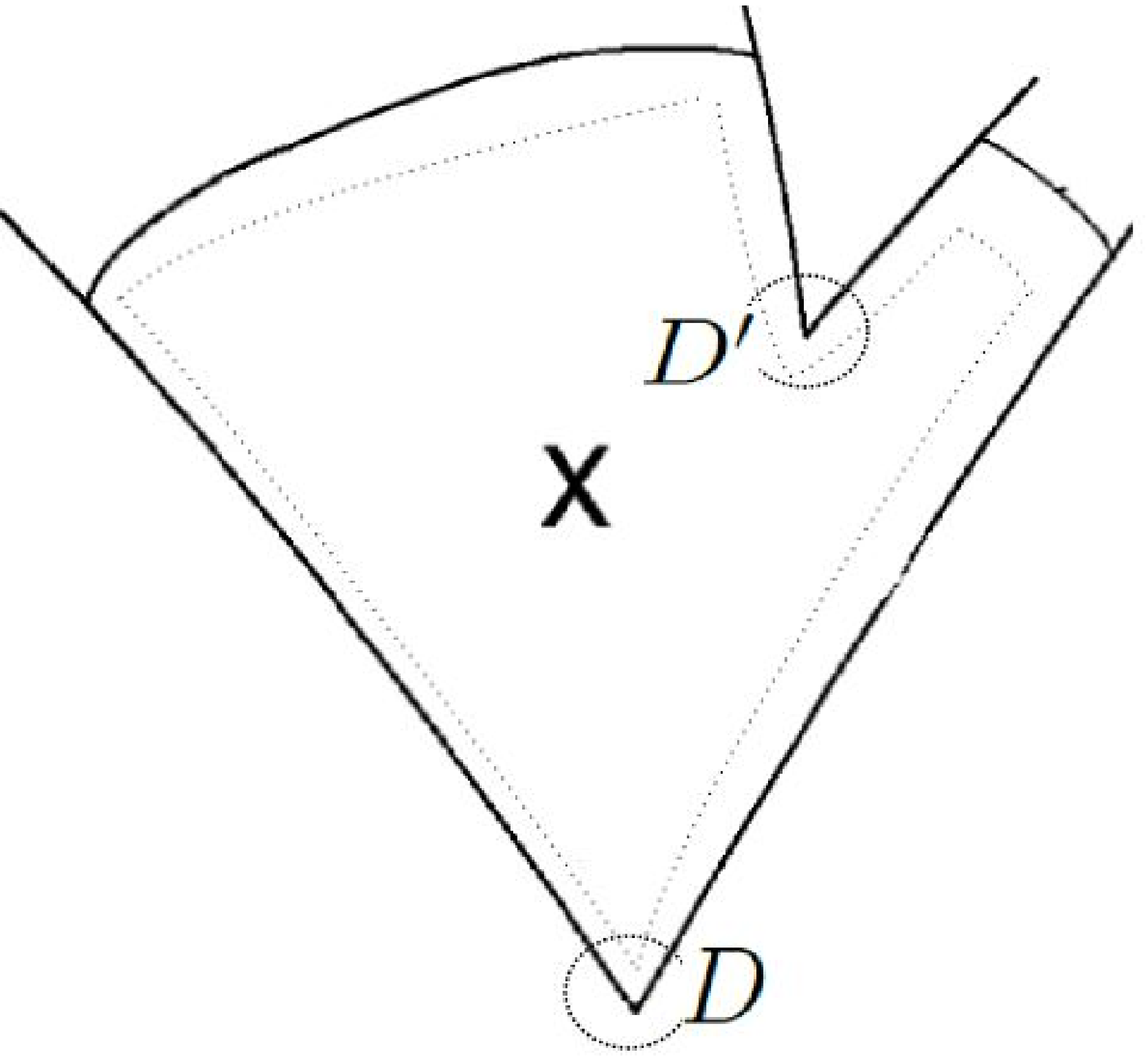}
\includegraphics[width=0.5\textwidth]{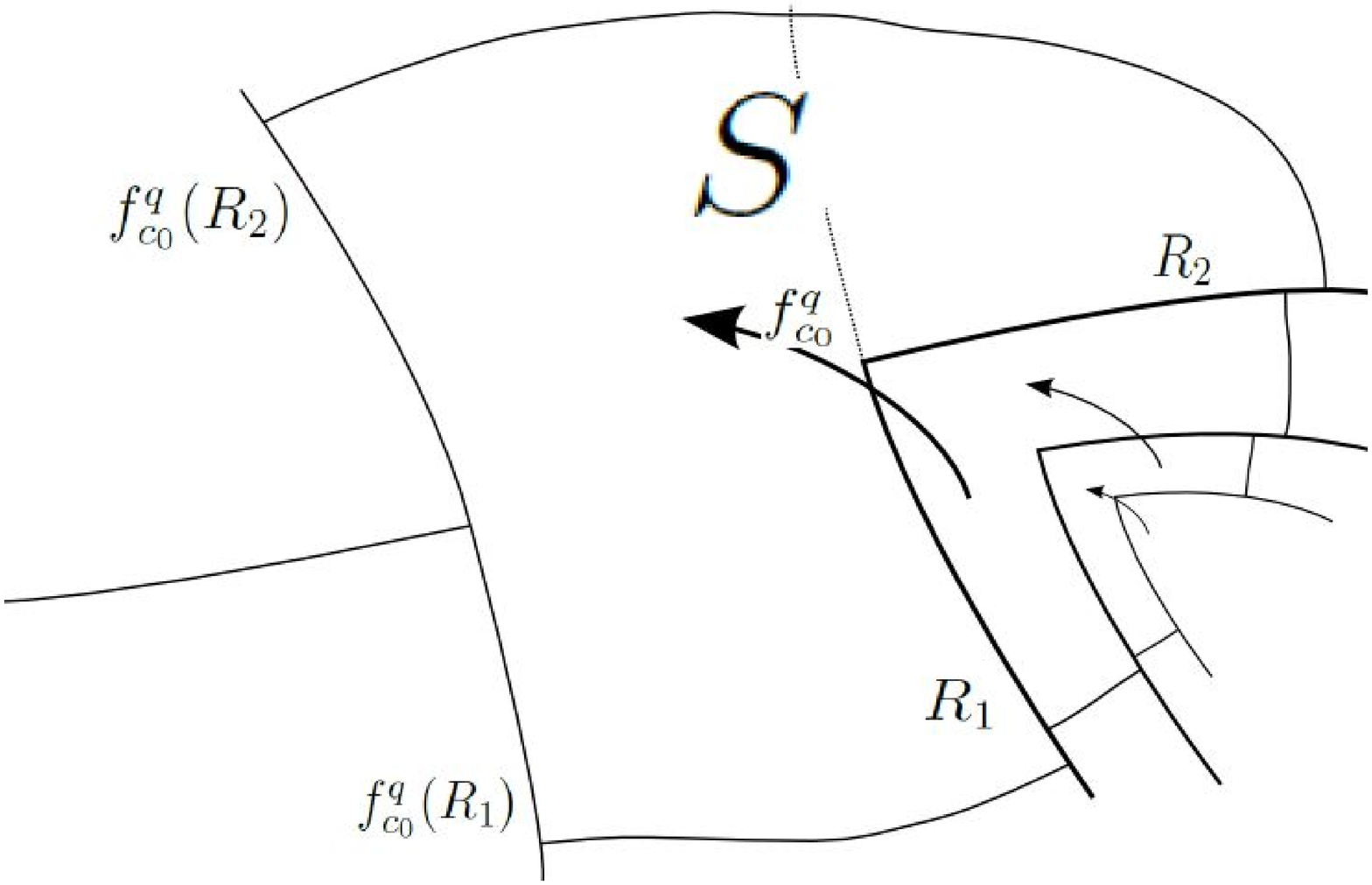}
  \caption{Illustration of the proof of Lemma~\ref{lem4}. There
are two vertices in $\partial X$; let $D$ and $D'$ be $\eta$-disks
around these two vertices. All rays sufficiently close to the
boundary rays of $X$ have their entire limit sets within $D$ or $D'$
(because the vertices of $X$ are repelling periodic and preperiodic
points and have trivial fibers). Therefore, these rays, together
with equipotentials close to those on the boundary of $X$, cover a
definite neighborhood of $\partial X\setminus(D\cup D')$. Right: The
construction of the strip $S$.} \label{Fig:LyubichLemmaPart1}
\end{figure}

The only case left is when $V^{n}$ is very close to one of the two
vertices of $\partial X$. The small Julia set corresponding to $X$
has two fixed points; we call them $\alpha$ and $\beta$ (in analogy
to standard notation) so that $\alpha$ is the dividing fixed point
of the small Julia set. The non-dividing fixed point $\beta$ is the
periodic vertices of $X$; denote the non-periodic vertex by
$\beta'$. Let $q$ be the period of renormalization of $X$. We may
assume that $V$ is very close to $\beta$ (possibly by replacing $V$
with $f^q(V)$).

Denote by $\alpha'\in X$ the non-periodic preimage of $\alpha$ under
$f_{c_0}^q:X\rightarrow f_{c_0}^q(X)$. Let $R_1$ and $R_2$ be the
two rays that land at $\alpha'$ that separate $\beta$ from all other
rays landing at $\alpha'$ (if any). These rays have the following
two properties:
   \begin{itemize}
\item the rays $R_1$, $ R_2$ separate $\{\beta\}$ from $\{\beta',\alpha,c_0\}$;

\item

$f_{c_0}^q(R_1)$, $f_{c_0}^q(R_2)$ land at $\alpha$ and separate
$\{\beta, \alpha' \}$ from $\{\beta', c_0\}$.
\end{itemize}

Let $S\subset X$ be the strip bounded by the two ray pairs $R_1$ and
$R_2$ as well as $f_{c_0}^q(R_1)$ and $f_{c_0}^q(R_2)$. Then there
exists a $k>0$ so that $f_{c_0}^{qk}(V^{n})\subset S$; if we assume
$k$ to be minimal with that property, then $f_{c_0}^{ql}(V^{n})$ is
contained in the same component of $X\sm S$ as $\beta$
 for all $l=0,1,\dots,k-1$. (Note that the only part of this construction that depends on $n$ is $k$.)

Let $S'$ be $S$ truncated at some equipotential, say at the same
height as $V$; we have $f_{c_0}^{qk}(V^{n})\subset S'$. There is an
annulus $A\subset X$ for which $S'$ is the bounded complementary
component, and $c_0$ is contained in the unbounded complementary
component; it can be chosen so that $A$ has a conformal preimage
under $f_{c_0}^{q}$ that is separated from $c_0$ by the ray pair
$f_{c_0}^q(R_1)$, $f_{c_0}^q(R_2)$. Note again that $A$ does not
depend on $n$.

Pulling back this annulus under $f_{c_0}^{qk}$ along the orbit of
$V^{n}$, we obtain the annulus around $V^{n}$ with the same modulus.
\end{proof}

\subsection{The Unsimple Case}
\label{Sub:Unsimple} Let us fix  $X$, $V$, and $\VV$ as in
Lemma~\ref{lem4}. By $X'$ we denote the pullback of $X$ under
$f_{c_0}$, so that $f_{c_0}:X'\rightarrow X$ is two to one. We will
work in $X'\ni 0$ so that the critical value is not in the way of
further pull-backs.

Let $V^{k}\subset X'$ be a maximal pullback of $V$. Then
$X'\backslash V^{k}$ is an annulus, and its boundary moves
holomorphically whenever $c\in \intr  \VV$. By the $\lambda$-lemma
we have a holomorphic motion $h^k_c$ of the annulus $X'\backslash
V^{k}$ with the quasiconformal dilatation depending on the distance
of $c$ to $\partial\VV$.

\begin{proposition}[Parameter Annuli in Unsimple Case]
\label{prop3} For the parameter $c_0$, let $Z^n_i$ be a puzzle piece
corresponding to an unsimple decoration and let $Z'^m_j$ be the
secondary puzzle piece with $Z'^m_j\supset Z^n_i$. Let $V\Subset X$
be as in Lemma~\ref{lem4}. If $Z'^m_j\Subset V$, then
$\ZZ'^m_j\Subset\VV$, and there are

\begin{itemize}
\item
an open Jordan disk $\JJ$ containing $\ZZ'^m_j$ and $c_0$
\item
a domain $V^k$
\item
and, for every $c\in\JJ$, an annulus $A^n_c$ the boundary of which
depends holomorphically on $c$
\end{itemize}
so that  $A^n_c\subset Z'^m_j$ and $Z^n_i$ is contained in the
bounded complementary component of $A^n_c$ whenever $Z^n_i$ exists,
and $f_c^{m+n'}\colon A^n_c\to X'\sm V^k$ is a conformal isomorphism
for some $n'\ge 0$.

For all $c\in\JJ$, there is a holomorphic motion from the closed
annulus $\ovl{X'\sm V^k}(c_0)$ to the closed annulus $\ovl{X'\sm
V^k}(c)$, and its dilatation is bounded in terms of the conformal
distance from $c$ to $\partial\VV$. This holomorphic motion can be
pulled back conformally, with the same dilatation, to a holomorphic
motion from $A^n_{c_0}$ to $A^n_c$.

For all $c\in\JJ$, the annuli $A^n_c$ are bounded by dynamic rays at
the same angles, and by equipotentials at equal heights. The
corresponding parameter rays at the same angles, and equipotentials
at the same heights, bound an annulus $\AA^n$ in parameter space.
Its modulus is bounded below by the distance of $\AA^n$ to
$\partial\VV$.
\end{proposition}

\begin{proof}
   Let us consider $$\JJ=\intr\VV\backslash \bigcup_{q<m}\bigcup_t\lfloor\ZZ'^{q}_t\rfloor \;.$$
   It is clear that $\ZZ'^m_j \subset \JJ$ (because all $Z'^m_j$ are maximal) and $\JJ$ is a Jordan disc.
We have $Z'^m_j\subset V\subset X$ for $c_0$ and thus for $c\in\JJ$
by construction, hence $\ZZ'^m_j\Subset\VV$.

We have $f^s_{c_0}\left(Z^n_i\right)\subset Z'^{m-s}_{j(s)}$ for
$s\le m$. By maximality of $Z'^m_j$, we have $f^s_{c_0}(Z'^m_j)\cap
Z'^m_j=\emptyset$ for $s\le m$; but by definition of ``unsimple'',
there is an $n''\le n-m$ such that
\begin{equation}
\label{eq3.1} f_{c_0}^{m+n''}\left(Z^n_i\right)\subset
   Z'^m_j \subset V\subset X\;.
\end{equation}
Therefore, $f^{m+n''-1}_{c_0}(Z^n_i)\subset X'$; let $n'$ be minimal
so that $f_{c_0}^{m+n'}\left(Z^n_i\right)$ has non-empty
intersection with $\intr X'$. Hence there exists a maximal pull-back
$V^k$ of $V$ so that
\begin{equation}
f_{c_0}^{m+n'}\left(Z^n_i\right) \subset V^k\subset X'
\label{Eq:MaxPullbackVk}
\end{equation}
(in fact, $k\le n''-n'+1$: the pull-back $V^{n''-n'+1}$ always
satisfies (~\ref{Eq:MaxPullbackVk}), and the maximal pull-back may
have smaller value of $k$).

We will now construct open annuli $A^n_c$ for all $c\in\JJ$ so that
$f_c^{m+n'}\colon A^n_c\to (\intr X')\sm V^k$ is a conformal
isomorphism. We will describe the construction for $c_0$ explicitly,
but the rays and equipotentials that define these annuli exist for
all $c\in\JJ$.

\begin{figure}
  \includegraphics[width=8cm]{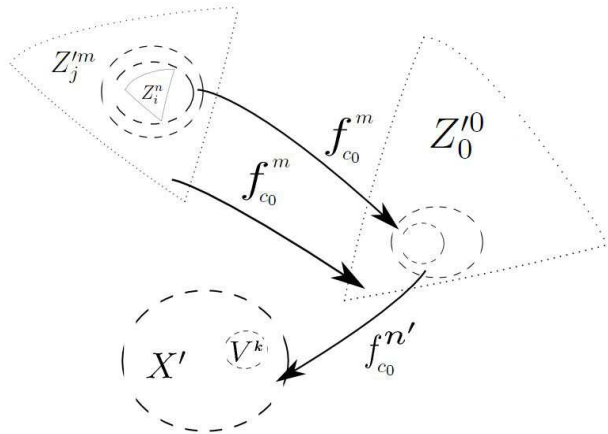}
  \caption{
  Construction of the annulus $A^n_{c_0}$. }
\end{figure}

For the parameter $c_0$, there exists a domain $X''\supset
f^m_{c_0}(Z^n_i)$ so that $f^{n'}_{c_0}\colon X''\to X'$ is a
conformal isomorphism, by minimality of $n'$. This domain is bounded
by certain dynamic rays and equipotentials, and an analogous domain
$X''$ thus exists for all $c\in\intr\XX$. Similarly, a domain
$V''\subset X''$ with $f_c^{n'}\colon V''\to V^k$ exists for all
$c\in\intr\VV$ (because $V^k$ exists for $c\in\VV$). (Note that for
some $c\in\intr\VV$, the puzzle piece $Z^n_i$ may not exist; but if
it does, then $f^m_c(Z^n_i)\subset V''$ because this is so for
$c_0$, by (~\ref{Eq:MaxPullbackVk})).

Now we have annuli $\intr(X'')\sm V''$, and we want to pull them
back $m$ more iterations. This will work for all $c\in\JJ$. Indeed,
for these $c$, the set $Z'^m_j$ exists, and $f^m_c(Z'^m_j)=Z'^0_0$.
For the parameter $c_0$, the puzzle piece
$Z'^0_0=f^m_{c_0}(Z'^m_j)\supset f^m_{c_0}(Z^n_i)$ intersects $X''$.
The way $X$ was constructed, it follows that $X''\subset Z'^0_0$
(this is the second condition in Lemma~\ref{lem4}). For all
$c\in\intr\XX$, the combinatorics of the boundaries of $X''$ and of
$Z'^0_0$ are the same, so these properties remain true for all
$c\in\intr\XX$. For every $c\in\JJ$, we have a conformal isomorphism
$f^m_c\colon  Z'^m_j\to  Z'^0_0$, and this yields an open annulus
$A^n_c\subset Z'^m_j$ so that $f^{m+n'}_c\colon A^n_c\to
\intr(X')\sm V^k$ is a conformal isomorphism.

We have $f^{m+n'}_{c_0}(Z^n_i)\subset V^k$, so $Z^n_i$ is contained
in the bounded complementary component of $A^n_{c_0}$. This property
persists for all parameters $c\in\JJ$ for which $Z^n_i$ exists.

The outer boundary of $X'$ consists of pieces of eight dynamic rays
and four equipotentials, and the same is true for the inner
boundary, which is $\partial V^k$. The boundary thus depends
holomorphically on $c$. As before, by the $\lambda$-lemma this
yields a holomorphic motion from $X'(c_0)\sm V^k(c_0)$ to $X'(c)\sm
V^k(c)$ the dilatation of which is bounded above by the distance of
$c$ to $\partial\VV$. Since all pull-backs were conformal, we obtain
a holomorphic motion from $A^n_{c_0}$ to $A^n_c$ the dilatation of
which is bounded again by the conformal distance of $c$ to
$\partial\VV$. Note that this is independent of $m$ and thus of $n$
(even though $\JJ$ depends on $m$).

Recall from Lemma~\ref{lem4} that the modulus of $X'\sm V^k$, and
thus of $A^n_{c_0}$, is bounded below by some constant $\eps_0/2$
that depends only on $X$ and $V$, and thus on $c_0$ alone but not on
$n$, $m$, or $k$.

As before, there is thus an annulus $\AA^n:=\{c\in\C\colon c\in
A^n_c\}$ in parameter space. The modulus of $\AA^n$ depends on $c_0$
and on the conformal distance from $\AA^n$ to $\partial\VV$.

This concludes the proof.
\end{proof}

\begin{cor}
For large $n$, the moduli of $\ZZ'^m_j\sm \ZZ^n_i$ are bounded below
by a constant that depends only on $c_0$.

More precisely, if $\ZZ'^m_j\Subset \VV$, then $\ZZ'^m_j\setminus
\ZZ^n_i\supset\AA^n$, and hence the modulus of the annulus
$\ZZ'^m_j\setminus \ZZ^n_i$ is bounded below in terms of its
conformal distance to $\partial\VV$.
\end{cor}

\begin{proof}
In Proposition~\ref{prop3}, we proved that $A^n_c\subset Z'^m_j$ and
that $Z^n_i$ is contained in the bounded complementary component of
$A^n_c$. Therefore $\AA^n\subset\ZZ'^m_j$, and $\ZZ^n_i$ is
contained in the bounded complementary component of $\AA^n$.
Therefore $\bmod(\ZZ'^m_j\setminus \ZZ^n_i)\ge \bmod\AA^n$. As $n$
tends to $\infty$, this conformal distance is bounded below, so that
all $\AA^n$ have their moduli bounded below.
\end{proof}

\end{document}